\documentclass[11pt,reqno]{amsart}

\usepackage[a4paper, left=30mm,right=30mm,top=30mm,bottom=30mm,marginpar=25mm]
{geometry}
\usepackage[english]{babel}
\usepackage[dvipsnames]{xcolor}
\usepackage[utf8x]{inputenc}
\usepackage{setspace}
\usepackage{enumitem} 
\usepackage{amsmath}
\usepackage{mathtools}
\usepackage[normalem]{ulem}
\usepackage{amsfonts}
\usepackage{amssymb}
\usepackage{fancyhdr}
\usepackage{textcomp}
\usepackage{graphicx} 
\usepackage{float}
\usepackage{listings}
\usepackage{amssymb}
\usepackage{color}

\newtheorem{theorem}{Theorem}[section]

\newtheorem{corollary}{Corollary}[section]
\newtheorem{proposition}{Proposition}[section]
\newtheorem{definition}{Definition}[section]
\newtheorem{remark}{Remark}[section]

\newcommand{\be}{\begin{equation}}
\newcommand{\ee}{\end{equation}}
\newcommand{\del}{\partial}
\newcommand{\bp}{\begin{proof}}
\newcommand{\ep}{\end{proof}}
\newcommand \RR   {\mathbb{R}}
  
\newcommand \Ga   {\Gamma}

\definecolor{listinggray}{gray}{0.9}
\definecolor{lbcolor}{rgb}{0.9,0.9,0.9}

\allowdisplaybreaks
\def\del{\partial}
\usepackage{xcolor}
\usepackage{soul}

\newenvironment{dedication}
  {
  }
  {
  }



\begin{document}
	
	\numberwithin{equation}{section}
	
	\title[Corrugated Versus Smooth Uniqueness and Stability of  Isometric Immersions]{Corrugated Versus Smooth Uniqueness and Stability of Negatively Curved Isometric Immersions}

	\author[C. Christoforou]{Cleopatra Christoforou}
	\address[Cleopatra Christoforou]{Department of Mathematics and Statistics,
		University of Cyprus, Nicosia 1678, Cyprus.}
	\email{christoforou.cleopatra@ucy.ac.cy} 
	
	\date{\today}

			\maketitle

\begin{dedication}
\begin{center} {\it \small Dedicated to my advisor Constantine Dafermos on the occasion of his 80th birthday.}
\end{center}
\end{dedication}

	\begin{abstract} We prove uniqueness of smooth isometric immersions within the class of negatively curved corrugated two-dimensional immersions embedded into $\mathbb{R}^3$. The main tool we use is the relative entropy method employed in the setting of differential geometry for the Gauss-Codazzi system. The result allows us to compare also two solutions to the Gauss-Codazzi system that correspond to a smooth and a $C^{1,1}$ isometric immersion of not necessarily the same metric and prove continuous dependence of their second fundamental forms in terms of the metric and initial data in $L^2$.
	\end{abstract}

	
\section{Introduction}\label{sec1}

	The isometric immersion problem is a fundamental and conceptually important problem in Differential Geometry and in the recent years, there is an extensive effort in examining the connection between Continuum Physics and isometric immersions from the research community in order to further develop each field adapting tools from either one and even combining them. As a consequence, some of these studies resulted in establishing the existence of the so-called \emph{corrugated surfaces}, which are 2-d isometric immersions embedded into $\RR^3$  of negative Gauss curvature that are not necessarily smooth but wrinkled, grooved and with edges and corners. The topic of this paper is the uniqueness of  \emph{corrugated surfaces} versus smooth ones and their stability using as a tool the relative entropy method in the setting of geometry.
	
	The history of the isometric immersion problem goes back to Riemann, when he  introduced abstract manifolds with a metric structure, the so-called Riemannian manifolds, by generalizing the classical objects of curves and surfaces in $\mathbb{R}^3$. Naturally, this gave rise to the isometric embedding question, which is the issue whether a Riemannian manifold can be immersed into some Euclidean space $\mathbb{R}^N$ with its given induced metric. 
	Nash in 1954 and Kuiper in 1955 proved the existence of a global $C^1$ isometric embedding of $n-$ dimensional Riemannian manifolds in $\mathbb{R}^N$ for $N=2n+1$. Actually, later on, Nash was able to establish smooth embedding into $\mathbb{R}^N$ by using smoothing operators that led to the famous 
	Nash-Moser iteration.
	Clearly, the issue for the smallest possible $N$ was raised and this remains an open question up to today.

The central topic of this paper is the isometric immersion problem for  2-d negatively curved manifolds embedded into $\mathbb{R}^3$, which can be formulated as an initial or initial-boundary value problem for a system of nonlinear partial differential equations of mixed elliptic--hyperbolic type, the so-called \emph{Gauss--Codazzi system}. This system is governed by  two partial differential equations
	\begin{equation*}
\begin{aligned}
	\del_1 M-\del_2 L&=\Ga_{22}^{(2)}L-2\Ga_{12}^{(2)}M+\Ga_{11}^{(2)} N\\
\del_1 N-\del_2 M&=-\Ga_{22}^{(1)}L+2\Ga_{12}^{(1)}M-\Ga_{11}^{(1)} N
\end{aligned}
\end{equation*}
and a constrained equation
	\begin{equation*}
LN-M^2=K 
\end{equation*}
with the unknown variables $(L,M,N)$ corresponding to the second fundamental form of the manifold.
This system is in general of mixed-type and its type is related to the sign of the Gauss curvature $K$.  In particular, for positively curved surfaces, the system is elliptic and the theory has been well developed in this case. In contrast, for negatively curved surfaces, the system is hyperbolic and there are still several questions open about global or local issues. The investigation of negatively curved surfaces was initiated by Hilbert, who proved that complete manifolds with Gauss curvature bounded above by a negative constant cannot be isometrically embedded in $\mathbb{R}^3$. Efimov in 1963 came to the same nonexistence conclusion if the Gauss curvature is bounded away from zero at a specific rate. 
The first positive result came from Hong~\cite{H} in 1993, when he showed that there exists a complete negatively curved manifold isometrically embedded into $\mathbb{R}^3$ if the Gauss curvature is negative and decays at a certain rate at infinity to zero (faster than $\rho^{-2}$ as $\rho\to\infty$). Note that $\rho$ here stands for the radius in polar coordinates in $\mathbb{R}^2$. All these results involved only {\emph smooth} immersions up to this point. 

The establishment of global non-smooth immersions has been initiated by Chen et al~\cite{CSW} and later followed by works of other researchers in the community of conservation laws that gave rise for instance to~\cite{CSW, CHW, CS2015, Cisom, CHW2} under various conditions on the given negative Gauss curvature related to its decay rate to zero and for various rough data (in $L^\infty$ space or the space of bounded variation $BV$).
The objects under consideration are two dimensional Riemannian manifolds, i.e. surfaces, with negative Gauss curvature that are not necessarily smooth, and they are called \emph{corrugated surfaces} having alternating grooves and ridges. At this point, it would be good for the reader to have in mind a very common example that of paper folding, how a sheet of paper can be folded. It can be bended, folded, or crumpled but cannot be compressed or stretched. Instead of considering surfaces with zero curvature, which include a sheet of paper, theses works~\cite{CSW, CHW, CS2015, Cisom, CHW2} resulted on the existence of \emph{corrugated surfaces}, i.e. of global embedding of two-dimensional manifolds with discontinuous data and negative Gauss curvature. For example in Figure~\ref{figure1}, one can see on the left the well-known catenoidal shell of revolution and on the right a corrugated surface with non circular cross-sections as drawn at $x=-x_0$ and of negative curvature, but both surfaces have the same metric $g$ of the catenoid. The main strategy followed in theses papers that was the fundamental idea in~\cite{CSW} is to treat the Gauss-Codazzi system as a special system of balance laws adapting construction schemes of weak solutions developed for balance laws in the setting of embedding a manifold into $\RR^3$. This novel approach provided connection between Continuum Mechanics and Differential Geometry and resulted in establishing the existence of the so-called \emph{corrugated surfaces}. 
In particular, in ~\cite{CS2015}, Christoforou and Slemrod used the compensated compactness method and captured Hong's result for $C^{1,1}$ isometric immersions. More precisely, they proved global existence of a $C^{1,1}$ negatively curved isometric immersion having Gauss curvature with the same rate $\rho^{-2-\delta}$ as Hong~\cite{H} with $\delta\in(0,2)$ and allowing initial data in $L^\infty$. See also~\cite{LS} for the construction of a mapping that sends every 2-d surface to solutions of the steady compressible Euler equations on the surface.
\begin{figure}[htbp]
\includegraphics[scale=0.5]{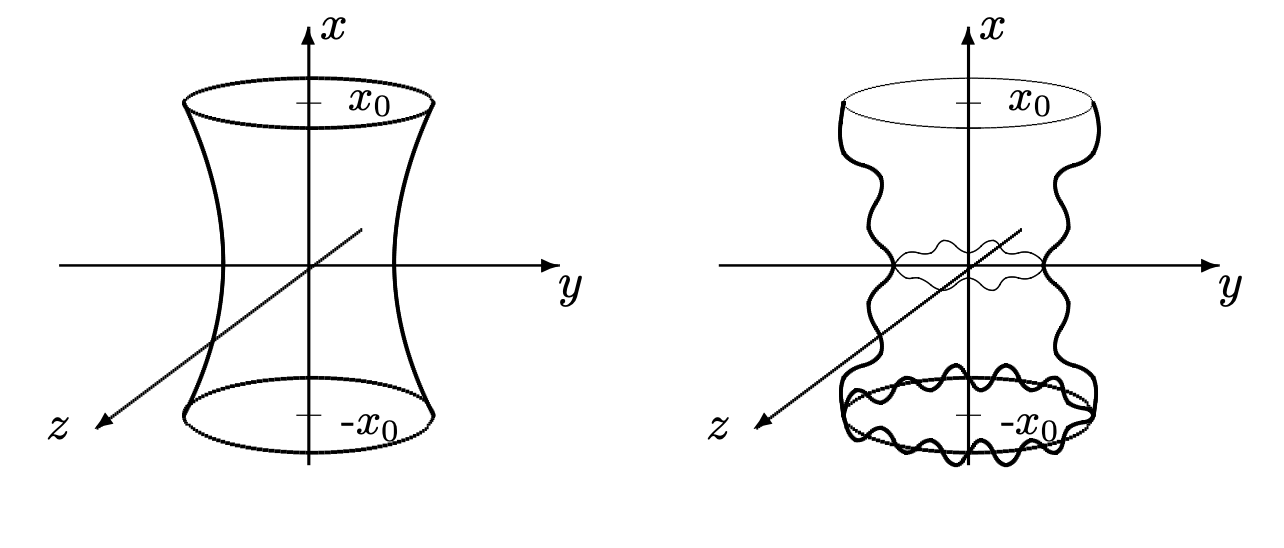}
\caption{A corrugated surface having the same metric as the catenoid in $\mathbb{R}^3$. \label{figure1}}
\end{figure}

The objective of this paper is to compare the smooth immersion of Hong~\cite{H} with the \emph{corrugated} one of Christoforou--Slemrod~\cite{CS2015}  by applying another methodology developed for Continuum Mechanics, that of the \emph{relative entropy method} but in the setting of \emph{the isometric immersion problem} with the goal to examine the uniqueness of smooth immersions in the class of corrugated ones as well as the stability of \emph{corrugated surfaces} w.r.t. their metric. This approach has been used extensively in Continuum Mechanics and it is connected with thermodynamics from its origins, but it has not been used so far in the setting of isometric embeddings. Our aim is therefore twofold: first, the application of this tool called \emph{relative entropy} in the setting of geometry that is performed for the first time and second, the proof of stability and uniqueness results mentioned above. As a consequence of our work, the \emph{relative entropy method} would serve as a new tool in the setting of immersions but also the connection between these two fields would be strengthened.

The \emph{relative entropy method} was  introduced by Dafermos \cite{dafermos79,dafermos79b} and DiPerna \cite{diperna79}  in the context of thermodynamics as a powerful tool in comparing solutions. It has been employed successfully for systems of conservation laws  ({\it e.g.}~\cite{diperna79,bds11,dst12,SV14}), or balance laws ({\it e.g.}~\cite{tzavaras05,MT14}) and even in the hyperbolic-parabolic setting. It has been used to prove measure valued weak versus smooth uniqueness and stability theorems as well as showing convergence results in the zero-viscosity or relaxation limits. An exposition of this method and its history can be found in the book~\cite{daf10}.
In~\cite{christoforou2016relative}, Christoforou and Tzavaras systematized  the derivation of relative entropy identities in a unifying framework and extended the class of systems for which this tool is applicable. The general framework of this method has also been  established by Christoforou in~\cite{Cinhom} for systems of balance laws with inhomogeneity, i.e. in this case the state, the flux and the source functions depend explicitly on the independent variables $(x,t)$. 
It is clear that the \emph{relative entropy method} was developed in parallel with the study of thermodynamics and therefore, this machinery is closely connected with the thermodynamical structure~\cite{CN1963}. Recalling the strong link between Continuum Physics and isometric immersions mentioned above, it is anticipated to employ the relative entropy method in the set up of the isometric immersion problem. In addition, the application of the relative entropy method as an outcome of this paper will build up and strengthen 
the connection across  \emph{fluid dynamics} and \emph{differential geometry} that has been developed in the recent years and also reveal the associated ``energy" of corrugated surfaces. 

The study 	of these objects, corrugated surfaces, has received a lot of attention the last decade from the community of conservation laws, in an effort to understand the appropriate notion of weak solutions to multi-d conservation laws. Actually, these  arise in non-uniqueness paradox in an attempt to capture phenomena present in multi-d systems and therefore, there is a need to study them and understand the characteristics of proper weak solutions excluding unwanted objects. Hence, the study of corrugated immersions is significant not only for the geometers, but also for its connections to multi-d conservation laws. We also expect that the notion of corrugated immersions would also arise in the applications of the isometric immersion problem in various disciplines, for instance in biology, big data processing and deep learning, and, hence, their mathematical interpretation will serve in the better understanding of these applied areas.
	
	The structure of the paper is the following: In Section~\ref{S2}, we set up the isometric immersion problem and present the available existence results for smooth and corrugated immersions corresponding to Hong's decay rate $\rho^{-2-\delta}$ and also formulate the problem in the setting of balance laws. Then, in Section~\ref{S3}, we  prove our main theorem, that includes the $L^2$ stability of corrugated versus smooth immersions with respect to initial data and the metric as well as the corrugated-smooth uniqueness. In this section, we perform the relative entropy calculations under appropriate hypotheses on the immersed surfaces and obtain continuous dependence estimates. We conclude the paper with two examples on helicoid-type surfaces and a comparison with Hong's immersion in terms of the metrics.

\medskip
	
\section{The Isometric Immersion Problem} \label{S2}
In this section, we introduce the isometric immersion problem, describe its history and current status in terms of negatively curved immersions. The aim is to set up the equations and formulate it as a system of balance laws.

Let $\Omega\subset\RR^2$ be an open set and  ${\bf y}:\Omega\to\RR^3$ a map such that the surface ${\bf y}(\Omega)\subset\RR^3$ has a tangent plane at every point ${\bf y}(x_1,x_2)$ spanned  by the vectors $\{\del_1 {\bf y},\del_2 {\bf y}\}$. Then, we immediately get that the unit normal vector ${\bf n}$ to the surface ${\bf y}(\Omega)$ is given by
$$
	{\bf n}=\frac{\del_1 {\bf y}\times\del_2 {\bf y}}{|\del_1 {\bf y}\times\del_2 {\bf y}|}
$$
with the corresponding metric 
\begin{align*}
ds^2=&d{\bf y}\cdot d{\bf y}\;\nonumber\\
=&(\del_1 {\bf y}\cdot\del_1 {\bf y})(dx_1)^2+2(\del_1 {\bf y}\cdot\del_2 {\bf y})dx_1\,dx_2+(\del_2 {\bf y}\cdot\del_2 {\bf y})(dx_2)^2\;.
\end{align*}
Now the question is given $(g_{ij})$ $i,j=1,2$ functions in $\Omega$, with $g_{12}=g_{21}$, whether there exists a map ${\bf y}:\Omega\to\RR^3$ so that
$$
 	d{\bf y}\cdot d{\bf y}=g_{11} (dx_1)^2+2g_{12} dx_1dx_2+g_{22}(dx_2)^2\;,
$$
or equivalently,
$$
\del_1{\bf y}\cdot\del _1{\bf y}=g_{11},\qquad \del_1{\bf y}\cdot\del _2{\bf y}=g_{12},\qquad \del_2{\bf y}\cdot\del _2{\bf y}=g_{22}
$$
with a linearly independent set $\{\del_1 {\bf y},\del_2 {\bf y}\}$ in $\RR^3$.
This \emph{inverse}  problem is the isometric immersion problem. By the last identities, it is obvious that this problem is fully nonlinear in terms of the three components of ${\bf y}$.

However the set up can be simplified if we bring into play the first and second fundamental forms of a manifold. In fact, a two dimensional manifold $(\mathcal{M},g)$ parametrized by $\Omega$ with associated metric $g=(g_{ij})$ admits the first fundamental form $I=(g_{i\, j})$ 
$$
	I\doteq g_{11} (dx_1)^2+2g_{12} dx_1dx_2+g_{22}(dx_2)^2\;,
$$
that is its metric $g$, and the second fundamental form $II=(h_{i\, j})$ 
$$
 II\doteq h_{11}(dx_1)^2+2h_{12} dx_1\,dx_2+h_{22}(dx_2)^2\;.
$$
If the manifold $(\mathcal{M},g)$ is isometrically immersed corresponding to a map ${\bf y}$, then it holds
$II=-d{\bf n}\cdot d{\bf y}$
with ${\bf n}$ being the unit normal vector to $\mathcal{M}$. Hence, the coefficients $(h_{ij})$ represent the orthogonality of ${\bf n}$ to the tangent plane and are associated with the second derivatives of ${\bf y}$.
Since ${\bf n}\cdot d{\bf y}=0$, it follows
$$
h_{11}= {\bf n}\cdot \del_1^2{\bf y}, \quad 2h_{12}={\bf n}\cdot\del_1\del_2{\bf y},\quad h_{22}={\bf n}\cdot\del_2^2{\bf y}.
$$
By equating the cross-partial derivatives of ${\bf y}$ and using the above identities, the isometric immersion problem reduces to solving for $(L,M,N)$ the Gauss--Codazzi system
\be\label{S2:GC}
\begin{array}{l}
\del_1 M-\del_2 L=\Ga_{22}^{(2)}L-2\Ga_{12}^{(2)}M+\Ga_{11}^{(2)} N\\
\del_1 N-\del_2 M=-\Ga_{22}^{(1)}L+2\Ga_{12}^{(1)}M-\Ga_{11}^{(1)} N
\end{array}
\ee
with the condition
\be\label{S2:GCrel}
 LN-M^2=K\,.
\ee
Here, the components of $(L,M,N)$ are given via the first and second forms as follows
\be\label{S2:hrelation}
 L=\frac{h_{11}}{\sqrt{|g|}},\qquad  M=\frac{h_{12}}{\sqrt{|g|}}, \qquad  N=\frac{h_{22}}{\sqrt{|g|}}\;,
\ee
while 
$|g|\doteq det(g_{ij})=g_{11} g_{22}-g^2_{12}$. Moreover, $K$ stands as usual for the Gauss curvature $K=K(x_1,x_2)$  of the manifold $(\mathcal{M},g)$ that is given by
\be\label{S2K}
	K(x_1,x_2)=\frac{R_{1212}}{|g|}\;,
\ee
where $R_{ijkl}$ is the curvature tensor 
$$
	R_{ijkl}=g_{lm}\left(\del_k\Ga_{ij}^{(m)}-\del_j\Ga_{ik}^{(m)}+\Ga_{ij}^{(n)}\Ga_{nk}^{(m)}-\Ga_{ik}^{(n)}\Ga_{nj}^{(m)}\right)\;,
$$
and $\Ga_{ij}^{(k)}$ is the Christoffel symbol
$$
	\Ga_{ij}^{(k)}\doteq \frac{1}{2} g^{kl}\left( \del_j g_{il}+\del_i g_{jl}-\del_lg_{ij}\right)\;.
$$
Here, the indices $i,\,j,\,k,\,l=1,2$, $(\del_1,\del_2)=(\del_{x_1},\del_{x_2})$ and the summation convention is used. Also, we use the standard notation that $(g^{kl})$ is the inverse of $(g_{ij})$. From the above expressions, we deduce, as Theorem Egregium states, that the Gauss curvature $K$ is determined from the metric $g$. Thus, in an isometric immersion problem that the metric $g$ is given, then $K$ and $\Ga_{ij}^{(k)}$ are  known functions. The Gauss-Codazzi system is in general of mixed elliptic--hyperbolic type and the type changes according to the sign of the Gauss  curvature.
The system is hyperbolic for  
surfaces with negative Gauss curvature and  
the corresponding problem can be formulated as an initial or initial--boundary value problem. 
Additionally, by the fundamental theorem of surface theory, we know that if $(g_{ij})$ and $(h_{ij})$ satisfy the Gauss-Codazzi system~\eqref{S2:GC}--\eqref{S2:hrelation}, then there exists a surface embedded into $\RR^3$ with first and second fundamental forms $I$ and $II$. This result holds in the smooth case and it has been extended up to the case when  $(h_{ij})\in L^\infty_{loc}(\Omega)$ for given $(g_{ij})\in W^{1,\infty}_{loc}(\Omega)$ due to S.~Mardare~\cite{M1}. In this case, the immersion ${\bf y}$ is $C^{1,1}(\Omega)$ locally. In short,
 the isometric immersion problem reduces to, given a positive definite metric $(g_{ij})\in W^{1,\infty}_{loc}(\Omega)$, solving the Gauss-Codazzi system~\eqref{S2:GC}--\eqref{S2:hrelation} for $(L,M,N)\in L^{\infty}_{loc}(\Omega)$ and hence, this immediately yields an isometric immersion ${\bf y}({\Omega})$, which is  $C^{1,1}$ locally. 
The reader can find an exposition of the surface theory  in the books~\cite{DoC,HH}.

For convenience of the reader, we state Mardare's result in full:
\begin{theorem}[S.~Mardare~\cite{M1}]\label{Mardare}
Assume that $\Omega$ is a connected and simply-connected open subset of $\RR^2$ and that the matrix fields $(g_{ij})\in W^{1,\infty}_{loc}(\Omega)$ being symmetric positive definite and $(h_{ij})\in L^\infty_{loc}(\Omega)$ symmetric satisfy the Gauss and Codazzi-Mainardi equations in $\mathcal{D}'(\Omega)$. Then there exists a mapping ${\bf y}\in W^{2,\infty}_{loc}(\Omega,\RR^3)$ such that
\begin{align}
g_{ij}&=\partial_i {\bf y}\cdot\partial_j {\bf y},\qquad
h_{ij}=\partial_{ij} {\bf y}\cdot\frac{\partial_1{\bf y}\times \partial_2 {\bf y}}{|\partial_1{\bf y} \times\partial_2 {\bf y}|}\nonumber
\end{align}
a.e. in $\Omega$. Moreover, the mapping ${\bf y}$ is unique in $W^{2,\infty}_{loc}(\Omega,\RR^3)$ up to proper isometries in $\RR^3$.
\end{theorem}

In this article, we restrict our attention to geodesically complete manifolds noting that this notion is equivalent to the Riemannian manifold defining a complete metric space. In fact, do~Carmo notes ``intuitively, this means that the manifold does not have any holes or boundaries" and here it is the definition:

\begin{definition} $(\mathcal{M},g)$ is a geodesically complete Riemannian manifold
if and only if every geodesic can be extended indefinitely.
\end{definition}


Now, by Han and Hong~{\cite[Lemma 10.2.1]{HH}}, if $(\mathcal{M},g)$ is a geodesically complete simply connected smooth two
dimensional Riemannian manifold with non-positive Gauss curvature, then there exists a global geodesic coordinate system $(x,t)$ in $\mathcal{M}$ with metric 
\be\label{S2 g} g=dt^{2}+(b(x,t))^{2}dx^{2}\ee
where $b=b(x,t)$ is a smooth positive function satisfying $
b(x,0)=1$ and $\partial _{t}b(x,0)=0$ for $x\in \RR$. This allows us to work on this geodesic coordinate system $(x,t)$ and then
the Gauss-Codazzi
system takes the form
\be\label{S2:GC new}
\begin{array}{l}
\partial _{t}L-\partial _{x}M=L\partial _{t}\ln b-M\partial _{x}\ln b+Nb\partial _{t}b,\\
\partial _{t}M-\partial _{x}N=-M\partial _{t}\ln b,
\end{array}
\ee
with
\be\label{S2:GC newdt} LN-M^{2}=Kb^{2},\ee
where $\partial _{tt}b=-Kb$ defines the Gauss curvature $K$ in terms of the
metric. 

As mentioned in the introduction, Hong's result~\cite{H} in 1993 first established that if the Gauss curvature is negative and decays faster than $t^{-2}$ as $t\to\infty$, then complete negatively curved manifolds that are smooth can be isometrically embedded into $\mathbb{R}^3$. On the other hand in~\cite{CS2015}, Christoforou and Slemrod 
captured Hong's result for $C^{1,1}$ isometric immersions, i.e. proved that there exists  $(L,M,N)\in L^\infty$ in $\RR^2$ at the same rate $t^{-2-\delta}$ of Gauss curvature as Hong~\cite{H}.
For the information of the reader, convergence of schemes  such as the random choice, front tracking or vanishing viscosity with dissipative source terms as well as the compensated compactness method have been used in~\cite{Cisom, CS2015, CSW,  CHW, CHW2} to construct \emph{corrugated surfaces} using machinery from balance laws. In this paper, we are interested in the immersions corresponding to the slower decay rate $t^{-2-\delta}$ for $\delta\in(0,2)$ of Hong, for which existence is known up to now. Therefore, we only state the results of Hong~\cite{H} and Christoforou--Slemrod~\cite{CS2015} below, respectively:

\begin{theorem}[Han and Hong~{\cite[Theorem 10.2.2]{HH}}]\label{HH}
Let $(\overline{\mathcal{M}},\bar g)$ be a complete simply
connected two dimensional Riemannian manifold  with negative Gauss curvature $\overline{ K}$
with metric $\bar g=dt ^{2}+(\bar b(x,t))^2 dx^{2}$. Assume for some constant $\delta>0$ that
\begin{enumerate}
\item[(i)] $t ^{2+\delta }\left\vert \overline{ K}\right\vert $ is decreasing in $|t|$, $|t|>T$;
\item[(ii)] $\partial _{x }^{i}\ln \left\vert \overline{ K}\right\vert$, for $i=1,2$ and $t \partial
_{t }\partial _{x }\ln \left\vert \overline{ K}\right\vert $ are bounded;
\item[(iii)] $\overline{ K}$ is periodic in $x$ with period $2\pi$.
\end{enumerate}
Then $(\overline{\mathcal{M}},\bar g)$ admits a smooth isometric immersion $\bar{\bf y}$ in $
\mathbb{R}
^{3}.$
\end{theorem}

\begin{theorem}[Christoforou and Slemrod~{\cite[Theorem 6.3]{CS2015}}]\label{thm2}
Let $(\mathcal{M},g)$ be a geodesically complete simply connected smooth two
dimensional Riemannian manifold with negative Gauss curvature $K$ and induced metric $ g=dt ^{2}+( b(x,t))^2 dx^{2}$.
For some $T_1>0$ large enough, assume that $b$ and $K$ are independent of $x$ satisfying 
\be\label{SV: h}
\partial_{tt} b=-K \,b,\quad b(0)=1,\quad \partial_{t} b(0)=0\;,
\ee
and
\be\label{SV: k} |K|=\frac{C}{(1+|t|)^{2+{\delta}}},\quad \text{for all }|t|>T_1\ee
for some constant $C>0$ and $\delta\in(0,2)$. Then $(\mathcal{M}, g)$ admits a global isometric immersion ${\bf y}$ in $\RR^3$, which is locally $C^{1,1}(\RR^2)$.
\end{theorem}

Now, we further simplify system~\eqref{S2:GC new}--\eqref{S2:GC newdt}  by scaling the variables as follows:
\be \ell=\frac{L}{b^2\sqrt{|K|}},\quad m=\frac{M}{b\sqrt{|K|}},\quad n=\frac{N}{\sqrt{|K|}}\ee
to obtain
\be\label{SV:lm}
\begin{array}{l}
\partial _{t}\ell-\frac{1}{b}\partial _{x}m+(\ell-n)\partial _{t}\ln b+\frac{\ell}{2}\partial _{t}\ln |K|  -\frac{m}{2b}\partial _{x}\ln |K|=0,\\\\
\partial _{t}m-\frac{1}{b}\partial _{x}n+2m\partial _{t}\ln b+\frac{m}{2} \partial _{t}\ln |K| -\frac{n}{2}\partial _{x}\ln |K|   =0,
\end{array}
\ee
with
\be\label{SV:rel} \ell n-m^{2}=-1. \ee

In the geodesic coordinate system $(x,t)$, we treat $t$ as the time variable and $x$ as space. In fact, in~\cite{CS2015}, the existence of solutions to the Gauss-Codazzi system is established by assigning appropriate data at $t=0$ and solving the system forward and backwards in time. Then the two solutions are glued since the initial data to~\eqref{SV:lm} at $t=0$ are chosen to be at least $C^1$. Another point is that the Codazzi equation allows us to seek for solutions $(\ell,m)$ to the system of pdes and treating the component $n$ as a function of the unknowns. Indeed, assigning initial data
\begin{equation}
\label{S2: conslaws ID}
(\ell(x,0),m(x,0))=(\ell_0,m_0),\qquad x\in\RR
\end{equation}
and using a vanishing viscosity scheme, it is shown that
there exists a subsequence, $(\ell^\mu,m^\mu,n^\mu)$ that converges weak$\,^*$ in $L^\infty(\Omega)$ to $(\ell,m,n)$ as $\mu\to0$ such that $|(\ell,m,n)|\le A$ a.e. in $\Omega$
where $A$ is independent of the viscosity parameter $\mu$ and the limit $(\ell,m,n)$ is a bounded weak solution of the Gauss–Codazzi system in the domain  $\Omega$.

In addition, following the analysis in~\cite{CS2015}, we assume that $b$ is a function only of $t$, i.e. independent of $x$. This assumption is made only to avoid further technicalities that would give rise to additional source terms in the system of conservation laws written next. In other words, system~\eqref{SV:lm}--\eqref{SV:rel} is equivalent to
\begin{equation}
\label{S2: conslaws}
\partial _{t} U+\partial _{x}(f(U;b))+P(U;b)=0
\end{equation}
with the conserved quantity $U=(\ell,m)^T$, the flux $f$ and the source $P$ to be
\begin{equation}\label{S2: fP}
f(U;b)=-\frac{1}{b(t)}\left[\begin{array}{c}
m\\ \frac{m^2-1}{\ell}
\end{array}
\right],\qquad 
P(U;b)=\left[\begin{array}{c}
(\ell-\frac{m^2-1}{\ell})\partial _{t}\ln b+\frac{\ell}{2}\partial _{t}\ln |K|  \\ 2m\partial _{t}\ln b+\frac{m}{2} \partial _{t}\ln |K| 
\end{array}
\right],
\end{equation}
Actually, if the function $b$ depends also on $x$, then we can still write the system in conserved form with the flux $f$ depending explicitly also on $x$, but at the expense of having additional terms in the source $P$ that would involve space derivatives of $b$. Thus, from here and on the function $b$ associated with the metric $g$ via~\eqref{S2 g} is independent of $x$. We note that both the flux and the source are inhomogeneous in the sense that depend explicitly on time via this function $b$. This of course will affect our analysis below and therefore, we use the notation $\phi(\cdot; b)$ when a generic function $\phi$ depends also on $b$. Let us here point out that in~\cite{Cinhom}, we perform the relative entropy calculation when the constitutive functions depend explicitly on $x$ and $t$, which is more general than our case here, since the dependence is only via the component $b(t)$. Therefore, we avoid applying the results of~\cite{Cinhom} since the hypotheses placed there are not suitable for the setting given here.

The eigenvalues associated with system~\eqref{SV:lm} are
\be\lambda_1=-\frac{m}{b\,\ell} +\frac{1}{b\,\ell},\qquad\lambda_2=-\frac{m}{b\,\ell} -\frac{1}{b\,\ell}\ee 
and we see that each characteristic field is linear degenerate. System~\eqref{SV:lm} is \emph{strictly hyperbolic} if $\lambda_1<\lambda_2$, or equivalently if $\ell<0$ is finite.
It is also worth mentioning that the Riemann invariants associated with the above system are
\be\label{SV:uv} u=-\frac{m}{\ell}+\frac{1}{\ell},\qquad v=-\frac{m}{\ell}-\frac{1}{\ell}\;.\ee
Then, uniform strict hyperbolicity in the $(u,v)$ variables is equivalent to $v-u=-\frac{2}{\ell}$ being positive and uniformly bounded away from zero. That is again true if $\ell<0$ is finite.
In addition, the system is endowed with an entropy-entropy flux pair $(\eta,q)$ given by
\begin{equation}
\label{S2: eta-q}\eta(\ell,m)=-\frac{m^2+1}{\ell},\qquad q(\ell,m;b)=\frac{m^3-m}{b\,\ell^2}\;,
\end{equation}
and as we show below the entropy $\eta$ is convex for $\ell<0$.  Indeed the Hessian of $\eta$ is given by
\be\nabla^2\eta=\left[
\begin{array}{cc}
\eta_{\ell\ell} & \eta_{\ell m}\\
\eta_{\ell m} & \eta_{mm}
\end{array}\right]=-\frac{2}{\ell}\left[
\begin{array}{cc}
\frac{m^2+1}{\ell^2}  & -\frac{m}{\ell}\\
-\frac{m}{\ell} & 1
\end{array}\right].
\ee
 Since $v-u>0$, we have $\ell<0$ and we claim that the Hessian is positive definite. One can compute the determinant $\text{det} \nabla^2\eta=\frac{4}{\ell^4}$ and the eigenvalues $\mu_2>\mu_1>0$ of the Hessian $\nabla^2\eta$: 
\begin{align}
\mu_{1,2}&=-\frac{1}{\ell^3}(m^2+1+\ell^2)\pm\frac{1}{\ell^3}\sqrt{(m^2+1+\ell^2)^2-4\ell^2}
\nonumber\\
&=-\frac{1}{\ell^3}(m^2+1+\ell^2)\pm\frac{1}{\ell^3}\sqrt{m^4+2m^2(1+\ell^2)+(\ell^2-1)^2
}
>c>0\,.
\end{align}
It is immediate to conclude from the above expressions that the eigenvalues are real and strictly positive for $\ell<-r$, with $r>0$ a positive constant, and thus the Hessian $\nabla^2\eta$ is positive definite. Therefore, the Gauss-Codazzi system~\eqref{SV:lm}--\eqref{SV:rel} admits a strictly convex entropy $\eta$, i.e. there exists a constant $c_0>0$
\be\label{S2convex}
\xi^T\left( \nabla^2 \eta(\cdot) \right)\xi\ge c_0 |\xi|^2>0\qquad \forall\xi\in\mathbb{R}^n\setminus\{ 0\}.
\ee
It is well known of course that the success of the relative entropy method relies on the convexity of the entropy and therefore, we meet its applications mainly to hyperbolic systems with convex entropy, cf.~\cite{MR0285799}.
In addition, the convex entropy $\eta$ can be "viewed" as the energy of the immersion.
As a last comment at this point, we add that the solutions to the Gauss-Codazzi system established in Theorems~\ref{HH}--\ref{thm2} are uniformly bounded in $\RR^2$. Actually, from the analysis in~\cite{CS2015}, we have that the weak solution $U=(\ell,m)^T\in L^\infty(\RR^2)$ and there are bounded invariant regions with $u<0<v$. This yields that $U(x,t)$ is uniformly bounded for all $x$ and $t$ satisfying $U(x,t)\in B_0 $ with the ball
$B_0:=[-R_0,-r_0]\times(-1,1)$ for all $(x,t)\in\RR^2$ where $r_0$ and $R_0$ are positive constants. This follows from the $H^{-1}$ compactness and uniform bounds of the approximate sequence $(u^\mu,v^\mu)$ of Riemann invariants. Let us also denote by $\overline B:=[-\bar R,-\bar r]\times(-1,1)$ the range of the smooth solution $\bar U=(\bar \ell, \overline{m})^T$ obtained in Hong's  Theorem~\ref{HH}.

\section{The Relative Entropy Method for the Gauss-Codazzi System}\label{S3}
In this section, we consider two 2-dimensional Riemannian manifolds $(\mathcal{M},g)$ and $(\overline{\mathcal{M}},\bar g)$ that are geodesically complete simply connected with non-positive Gauss curvature and not having the same metric necessarily. Let $\Omega=\mathbb{R}\times [0,T]\subset \mathbb{R}^2$ be the common domain within which both manifolds are immersed surfaces. Furthermore, suppose that $(\overline{\mathcal{M}},\bar g)$ is a smooth isometric immersion with $\overline{\bf y}$ in $\Omega$ the corresponding smooth immersed map. On the other hand, let $(\mathcal{M},g)$ be an isometric immersion with its map ${\bf y}$ not necessarily smooth, but at least  locally $C^{1,1}(\Omega)$. For convenience, we call $(\mathcal{M},g)$ corrugated since smoothness is not required. Last, we consider the corresponding triplets $(\ell,m,n)$ and $(\bar\ell,\overline{m}, \bar n)$ satisfying the Gauss-Codazzi system~\eqref{SV:lm}--\eqref{SV:rel}. Eventually, we plan to apply our result in the global case that is: $\Omega=\mathbb{R}^2$, while $(\mathcal{M},g)$ and $(\overline{\mathcal{M}},\bar g)$ are the immersions obtained in Theorem~\ref{thm2} from~\cite{CS2015} and Theorem~\ref{HH} from~\cite{H}.

We aim to use the relative entropy method to compare the two manifolds and prove $L^2$ stability of the second fundamental form in terms of the initial data and the metric as well as a corrugated-smooth uniqueness theorem. Indeed, in what follows, we perform the relative entropy calculations comparing the two manifolds $(\mathcal{M},g)$ and $(\overline{\mathcal{M}},\bar g)$ at the level of the triplets $(\ell,m,n)$ and $(\bar\ell, \overline{m}, \bar n)$ using the connection of the immersion problem with Continuum Physics as described in the previous section.

For the purpose of performing the relative entropy calculations below, we set the following hypotheses:
\medskip

\noindent
{\bf Hypotheses (H)}. For some $T$, $0<T\le \infty$ and $\Omega=\mathbb{R}\times [0,T]\subset \mathbb{R}^2$, we assume 
\begin{enumerate}
\item[(H$_0$)] Boundedness of the solutions to the Gauss-Codazzi, i.e. $U=(\ell,m)^T\in B_0$ and $\overline U=(\bar \ell,\overline m)^T\in\overline B$  for all $(x,t)\in\Omega\subset\RR^2$ and $\overline U(\cdot,t)\in L^1(\RR)\cap BV(\RR)$ for all $t\in [0,T]$.
\item[(H$_1$)] The pair $K$ and $b$ are functions of  $(x,t)\in\Omega\subset\RR^2$ that are both independent of $x$. Moreover, the function $b>0$ is associated with the metric $g$ via~\eqref{S2 g} of a manifold $(\mathcal{M},g)$, while $K<0$ corresponds to the Gauss curvature of the same manifold, i.e.~\eqref{S2K}. Similar conditions hold for the pair $\overline K$ and $\bar b$ in the case of the smooth manifold $(\overline{\mathcal{M}},\bar g)$.
\item[(H$_2$)] The following integrability conditions hold: $$(\bar b(t))^{-1},\,\,   \partial _{t }\ln  b,\,\, \partial _{t }\ln \left\vert K\right\vert \in L^1[0,T]\,.$$
\item[(H$_3$)] The pairs $(K,b)$ and $(\overline K, \bar b)$ satisfy 
$$\left|\del_t\ln \frac{b(t)}{\bar b(t)}\right|\in L^2[0,T],\quad\,\left|\del_t\ln \frac{K(t)}{\overline K(t)}\right|\in L^2[0,T],\quad\, \left|\frac{1}{\bar b(t)}-\frac{1}{b(t)}\right| \in L^1[0,T]\,.$$
\end{enumerate}
\smallskip

Now, we consider the triplets $(\ell,m,n)$ and $(\bar\ell, \overline{m}, \bar n)$ that satisfy system~~\eqref{SV:lm}--\eqref{SV:rel} corresponding to $(\mathcal{M},g)$ and $(\overline{\mathcal{M}},\bar g)$. First, we define the relative entropy associated with system~\eqref{SV:lm}
\be\label{Sec3:relative eta}
\eta(\ell,m|\bar \ell, \overline{m}):=\eta(\ell,m)-\eta(\bar \ell, \overline{m})-\nabla\eta(\bar \ell, \overline{m})\left( U-\overline U\right)
\ee
that takes the special form
\begin{align}\label{Sec3:relative eta1}
\eta(\ell,m|\bar \ell, \overline{m})&=-\frac{1}{\bar\ell}\left\{(m- \overline{m})^2+(\ell-\bar\ell)\left(\eta(\ell,m)-\eta(\bar\ell, \overline{m}) \right)\right\}
\nonumber\\
&=-\frac{1}{\bar\ell}\left\{(m- \overline{m})^2+\frac{(1+\overline{m}^2)}{\ell\bar\ell}{(\ell-\bar\ell)^2}+\frac{1}{\ell}(m+\overline{m})(\overline{m}-m)(\ell-\bar\ell)
\right\}\;.
\end{align}
By the convexity of the entropy~\eqref{S2convex}, we have
$$
\eta(\ell,m|\bar \ell, \overline{m})\ge c_0(|\ell-\bar\ell|^2+|m-\overline{m}|^2)\;,
$$ for some $c_0>0$ constant. Actually the relative entropy is equivalent to the $L^2$ distance between the states $U=(\ell,m)^T$ and $\overline U=(\bar \ell, \overline{m})^T$. Indeed, since $U\in B_r$ and $\overline U\in \overline B$, we have
\be\label{Sec3:relative eta2}
 c_0(|\ell-\bar\ell|^2+|m- \overline{m}|^2)\le \eta(\ell,m|\bar \ell, \overline{m})\le c_1(|\ell-\bar\ell|^2+|m- \overline{m}|^2)\;,
\ee
for some constant $c_1>0$ possibly depending on the size of the balls $B_0$ and $\overline B$.

Next, we consider the relative entropy flux 
\be\label{Sec3:relative q}
q(\ell,m;b||\bar \ell, \overline{m};\bar b):=q(\ell,m;b)-q(\bar \ell, \overline{m};\bar b)-\nabla\eta(\bar \ell, \overline{m})\left( f(\ell,m;b)-f(\bar \ell, \overline{m};\bar b)\right)\;.
\ee
However, we observe here that the relative entropy flux depends also on the metric via the components $b$ and $\bar b$ associated with $(\mathcal{M},g)$ and the smooth manifold $(\overline{\mathcal{M}},\bar g)$, respectively.
For the purpose of performing the relative entropy computation, we need to introduce also the relative flux 
\be\label{Sec3:relative f}
f(\ell,m;b||\bar \ell, \overline{m};\bar b):= f(\ell,m;b)-f(\bar \ell, \overline{m};\bar b)-\nabla f(\bar \ell, \overline{m};\bar b)\left( U-\overline U\right)\;,
\ee 
and the relative gradient of the entropy
\be\label{Sec3:relative nabla eta}
\nabla\eta(\ell,m|\bar \ell, \overline{m}):=\nabla\eta(\ell,m)-\nabla\eta(\bar \ell, \overline{m})-\nabla^2\eta(\bar \ell, \overline{m}) \left( U-\overline U \right)\;.
\ee

Here it is the main theorem on stability and uniqueness of the immersions with respect to the metric and the initial data.

\begin{theorem}\label{newthm1} Let $0<T\le \infty$ and $\Omega=\mathbb{R}\times [0,T]\subset \mathbb{R}^2$.
Suppose that $(\mathcal{M},g)$ is a geodesically complete simply connected two
dimensional Riemannian manifold with negative Gauss curvature $K$ and induced metric $g$ of the form~\eqref{S2 g} 
with an associated corrugated isometric immersion ${\bf y}$ in $\RR^3$, which is locally $C^{1,1}(\Omega)$. Furthermore, suppose that $(\overline{\mathcal{M}},\bar g)$ is a geodesically complete simply connected smooth two
dimensional Riemannian manifold with negative Gauss curvature $\overline K$ and induced metric of the form~\eqref{S2 g} with an associated isometric immersion $\bar {\bf y}$ in $\RR^3$ that is smooth on $\Omega$. Under hypotheses (H), the followings hold true:
\begin{enumerate}
\item[(a)] Stability in $L^2$ of the second fundamental form in terms of the metric and the initial data in the sense
$$\int_{-\infty}^\infty |\ell(t)-\bar\ell(t)|^2+|m(t)-\overline{m}(t)|^2 dx\le e^{\Phi(t)}\int_{-\infty}^\infty |\ell_0-\bar\ell_0|^2+|m_0-\overline{m}_0|^2 dx+ C_g \int_0^T e^{\Phi(t)-\Phi(s)} \psi(s)ds  
$$
for some positive constant $C_g>0$ independent of the metrics and some bounded and positive functions $\Phi(s), \,\psi(s)$ on $[0,T]$, where $(\ell,m,n)$ and $(\bar\ell, \overline{m}, \bar n)$ are the corresponding solutions to~\eqref{SV:lm}--\eqref{SV:rel} with $(\ell_0,m_0)$ and $(\bar \ell_0, \overline{m}_0)$ the corresponding initial data~\eqref{S2: conslaws ID}.
\item[(b)]  If $b(t)=\bar b(t)$ for all $t\in [0,T]$, then
$$\int_{-\infty}^\infty |\ell(t)-\bar\ell(t)|^2+|m(t)-\overline{m}(t)|^2 dx\le C e^{\Phi(t)} \int_{-\infty}^\infty |\ell_0-\bar\ell_0|^2+|m_0-\overline{m}_0|^2 dx\;.$$
\item[(c)] If $b(t)=\bar b(t)$ for all $t\in [0,T]$ and $(\ell_0,m_0)=(\bar\ell_0,\overline{m}_0)$, then corrugated--smooth uniqueness holds true on $\Omega$, i.e. the corrugated immersion ${\bf y}$ coincides with the smooth one $\bar {\bf y}$.
\end{enumerate}
\end{theorem}
\bp
For convenience, we use the notation $\phi=\phi(U;b)$ and $\bar \phi=\phi(\overline U;\bar b)$ for some generic function $\phi$.

Let $U=(\ell,m)$ be a weak solution to~\eqref{SV:lm} and $\overline U=(\bar\ell, \overline{m})$ a strong one, with associated metric components $b$ and $\bar b$, respectively. In other words, it holds
\be\label{Sec3:Ueq}
\del_t U+\del_x(f(U;b))+P(U;b)=0
\ee
and
\be\label{Sec3:barUeq}
\del_t \bar U+\del_x(f(\bar U;b))+P(\bar U;\bar b)=0
\ee
using the form of conservation laws~\eqref{S2: conslaws}--\eqref{S2: fP} and initial data $(\ell_0,m_0)$ and $(\bar \ell_0, \overline{m}_0)$. Moreover, the entropy inequality holds
\be\label{Sec3:Ueqeta}
\del_t \eta(\ell,m)+\del_x(q(\ell,m;b))+\nabla\eta(\ell,m)\cdot P(U;b)\le 0\;,
\ee
while we have an identity in the smooth case, i.e.
\be\label{Sec3:Ueqetabar}
\del_t \eta(\bar \ell,\overline{m})+\del_x(q(\bar\ell,\overline{m}; \bar b))+\nabla\eta(\bar \ell, \overline{m})\cdot P(\overline U;\bar b)= 0\;.
\ee
Then
\be\label{Sec3:rel step1}
\del_t(\eta(\ell,m)-\eta(\bar \ell,\overline{m}))+\del_x(q(\ell,m;b)-q(\bar \ell,\overline{m};\bar b)) +\nabla\eta\cdot P -\nabla\bar\eta\cdot \overline P \le 0\;.
\ee
These relations allow us to compute
\begin{align}\label{Sec3:rel step2}
\partial_t&\left[\nabla\bar \eta\cdot(U-\overline U)\right]+\partial_x\left[ \nabla\bar \eta\,\cdot (f-\bar f)\right]=\nonumber\\
=& \nabla^2\bar \eta\,\partial_t\overline U\cdot (U-\overline U)+ \nabla^2\bar \eta\,\,\partial_x \overline U\cdot(f-\bar f)+ \nabla\bar \eta\,\cdot\left[\partial_t (U-\overline U) \partial_x (f-\bar f)\right]\nonumber\\
=& \nabla^2\bar \eta\,\left[-\left(\nabla \bar f\partial_x\overline U+ \bar P\right) \cdot(U-\overline U) +\partial_x \overline U \cdot(f-\bar f)
\right] -\nabla\bar \eta\,\cdot(P-\overline P)\nonumber\\
=&\nabla^2\bar \eta\, \partial_x \bar U\cdot  f(\ell,m;b|\bar \ell,\overline{m};\bar b) 
-\overline P\cdot \nabla^2\bar \eta\, (U-\overline U)-\nabla\bar \eta\,\cdot(P-\overline P) \;,
\end{align}
using \eqref{Sec3:Ueq},~\eqref{Sec3:barUeq} and~\eqref{Sec3:relative f}.
We also rewrite the expression
\begin{align}\label{Sec3:rel step3}
\nabla\eta\cdot P&-\nabla\bar\eta\cdot \bar P-\bar P\cdot\nabla^2\bar\eta (U-\overline U)-\nabla\bar\eta \cdot (P-\overline P)=\nonumber\\
&=(\nabla\eta-\nabla\bar\eta) \cdot (P-\bar P) +\overline P \cdot \nabla\eta(\ell,m|\bar \ell, \overline{m})\;,
\end{align}
and then combining~\eqref{Sec3:rel step1} and~\eqref{Sec3:rel step2}  with~\eqref{Sec3:rel step3}, we arrive at the relative entropy inequality
\begin{align}\label{Sec3:relentropyhyp}
\del_t(\eta(\ell,m|\bar \ell,\overline{m})&)+\del_x(q(\ell,m;b||\bar \ell,\overline{m};\bar b))+(\nabla\eta-\nabla\bar \eta)\cdot (P -\overline P) \le\nonumber\\
& \le -\nabla^2 \bar \eta\del_x \overline U\cdot  f(\ell,m;b||\bar \ell, \overline{m};\bar b)
-\overline P \cdot \nabla\eta(\ell,m|\bar \ell,\overline{m})\;.
\end{align}
Inequality~\eqref{Sec3:relentropyhyp} provides the evolution of the relative entropy and together with the convexity~\eqref{S2convex} consist of the core of our methodology that leads to the stability estimate in (a) as shown below.

From~\eqref{S2: fP} and~\eqref{Sec3:relative f}, we compute
$$
f(\ell,m;b||\bar \ell,\overline{m};\bar b)=\left(\frac{1}{\bar b(t)}-\frac{1}{b(t)}\right) \left[\begin{array}{c}
m\\ \dfrac{m^2-1}{\ell}
\end{array}\right]+\frac{1}{\bar b(t)}
\left[\begin{array}{c}
0\\ \eta(\ell,m|\bar \ell,\overline{m})+2\dfrac{(\ell-\bar \ell)^2}{\ell\bar\ell^2}
\end{array}\right]\;.
$$
Thus, by~\eqref{Sec3:relative eta2} and Hypothesis (H$_0$), we have
\begin{align}
\label{Sec3:relative f3}
\int_{-\infty}^\infty \big |\nabla^2 \bar \eta\del_x \overline U\cdot f(\ell,m;b||&\bar \ell,\overline{m};\bar b)\big|dx
\le
 C_1 \left|\frac{1}{\bar b(t)}-\frac{1}{b(t)}\right|\|\del_x \overline U\|_{L^1}+ \frac{C_0}{\bar b(t)} \int_{-\infty}^\infty \eta(\ell,m|\bar \ell,\overline m)dx\nonumber\\
&\le C_g \left|\frac{1}{\bar b(t)}-\frac{1}{b(t)}\right| + \frac{C_0}{\bar b(t)} \int_{-\infty}^\infty \eta(\ell,m|\bar \ell,\overline m)dx\;.
\end{align}
Here the positive constants $C_g=C_g(B_0, \bar B)$ and $C_0=C_0(B_0, \bar B)$ depend on the size of the balls $B_0$ and $\bar B$, i.e. the $L^\infty$ bounds of $U$ and $\overline U$. Also, $C_g$ depends on the $\|\del_x  \overline U\|_{L^1}$, while $C_0$ depends as well on the constant of convexity.  For convenience we use $C_g$ and $C_0$ as universal constants in what follows.

Next, from~\eqref{Sec3:relative nabla eta} and~\eqref{S2: eta-q}$_1$, we have
$$
\nabla\eta(\ell,m|\bar \ell,\overline m)=-\frac{1}{\bar\ell}\left[ \begin{array}{c}
-(\ell+\bar\ell)\frac{(1+\overline m^2)}{\ell^2\bar\ell^2}{(\ell-\bar\ell)^2}+\eta(\ell,m|\bar \ell,\overline m)\\
-\frac{2}{\ell}(m-\overline m)(\ell-\bar\ell)+\frac{2\overline m}{\ell\bar\ell}(\ell-\bar\ell)^2
\end{array}
\right]\;.
$$
Using the same bounds of $U$ and $\overline U$ and convexity, we immediately get that the relative gradient of the entropy is controlled by $\eta(\cdot |\cdot)$. In other words, we have
\be\label{Sec3:relative nabla eta2}
|\nabla\eta(\ell,m|\bar \ell,\overline m)|\le C_0 \eta(\ell,m|\bar \ell,\overline m)\;.
\ee
For last remaining source term in the evolution equation~\eqref{Sec3:relentropyhyp}, we compute
\begin{align*}
(\nabla\eta(\ell,m)-\nabla\eta(\bar \ell,\overline m))
&=\frac{2}{\ell\bar\ell} \left[\begin{array}{c}
\dfrac{ (m-\overline m)(m+\overline m){\bar\ell}^2+({\overline m}^2+1)(\ell+\bar\ell)(\bar\ell-\ell) }{2\ell\bar\ell}		\\ 
\bar m(\ell-\bar\ell)+(\overline m-m)\bar\ell
\end{array}
\right]
\end{align*}
and the two components of $P-\overline P=((P-\overline P)_1, (P-\overline P)_2)$ are
\begin{align*}
(P-\overline P)_1
=& 
\left\{\left(\ell-\dfrac{m^2-1}{\ell}\right) -\left(\bar\ell-\dfrac{\overline m^2-1}{\bar \ell}\right) \right\} \partial _{t}\ln b+
\left(\bar\ell-\dfrac{\overline m^2-1}{\bar \ell}\right) \partial _{t}\ln \dfrac{b}{\bar b}	\\
&+\dfrac{(\ell-\bar \ell)}{2}\partial _{t}\ln |K|
+\dfrac{\bar \ell}{2}\partial _{t}\ln \dfrac{|K|}{|\overline K|} \;,	
\end{align*}
and
\begin{align*}
(P-\overline P)_2
&= 
2(m-\overline m) \partial _{t}\ln b+
2\overline m \partial _{t}\ln \dfrac{b}{\bar b}	+\dfrac{(m-\overline m)}{2}\partial _{t}\ln |K|
+\dfrac{\overline m}{2}\partial _{t}\ln \dfrac{|K|}{|\overline K|} 	\,.
\end{align*}
Hence, we have
\begin{align*}
 \left|(\nabla\eta(\ell,m)-\nabla\eta(\bar \ell,\overline m))\right.&\left.\cdot (P-\overline P)\right| \le
C_1' (|\ell-\bar\ell|^2+|m-\overline m|^2) (| \partial _{t}\ln b(t)| + |\partial _{t}\ln |K(t)|\, |)\\
&+C_2'  (|\ell-\bar\ell|+|m-\overline m|)\left(\bar\ell \left|\partial _{t}\ln \dfrac{b(t)}{\bar b(t)} \right|
+\overline m \left|\partial _{t}\ln \dfrac{|K(t)|}{|\overline K(t)|} \right|
\right)\;,
\end{align*}
with the constant depending also on the $L^\infty$ bounds of $U$ and $\overline U$.
Therefore, we get the estimate
\begin{align}\label{Sec3:relative nabla etaP}
\int_{-\infty}^\infty (\nabla\eta(\ell,m)-&\nabla\eta(\bar \ell,\overline m))\cdot (P-\overline P)dx\le
C_g\left(\left|\del_t\ln \frac{b(t)}{\bar b(t)}\right|^2+\left|\del_t\ln \frac{K(t)}{\overline K(t)}\right|^2\right)
\nonumber\\
&+ C_0\, \left(1+|\del_t\ln b(t)|+\left|\del_t\ln |K(t)\right| |\right) \int_{-\infty}^\infty \eta(\ell,m|\bar \ell,\overline m)dx \end{align}
using~\eqref{Sec3:relative eta2} and the $L^1(\RR)$ bound of $\overline U(t)$ for $t\in[0,T]$.
Thus, by integrating~\eqref{Sec3:relentropyhyp} and using estimates~\eqref{Sec3:relative f3}, ~\eqref{Sec3:relative nabla eta2} and~\eqref{Sec3:relative nabla etaP}, we get
\begin{align}\label{Sec3:relentropyhyp2}
\frac{d}{dt}(\int_{-\infty}^\infty \eta(\ell,m|\bar \ell,\overline m)dx) \le& C_0\, \left(1+\frac{1}{\bar b(t)} +|\del_t\ln b(t)|+\left|\del_t\ln |K(t)\right| |\right) \int_{-\infty}^\infty \eta(\ell,m|\bar \ell,\overline m)dx\nonumber\\
&+C_g\left(\left|\del_t\ln \frac{b(t)}{\bar b(t)}\right|^2+\left|\del_t\ln \frac{K(t)}{\overline K(t)}\right|^2+ \left|\frac{1}{\bar b(t)}-\frac{1}{b(t)}\right|\right)\;.
\end{align}
Setting now
\be\label{S3: phi}
\varphi(t):=C_0\left(1+\frac{1}{\bar b(t)} +|\del_t\ln b(t)|+\left|\del_t\ln |K(t)\right| |\right)
\ee
and
\be\label{S3: psi}
\psi(t):=\left|\del_t\ln \frac{b(t)}{\bar b(t)}\right|^2+\left|\del_t\ln \frac{K(t)}{\overline K(t)}\right|^2+ \left|\frac{1}{\bar b(t)}-\frac{1}{b(t)}\right|\;,
\ee
we see that $\varphi,\,\psi\in L^1[0,T]$, by Hypotheses (H). Applying now Gronwall's inequality, we immediately derive the desired estimate on the relative entropy
$$
\int_{-\infty}^\infty \eta(\ell(t),m(t)|\bar \ell(t),\overline m(t)) dx\le e^{\Phi(t)}\int_{-\infty}^\infty \eta(\ell_0,m_0|\bar \ell_0,\overline m_0) dx+C_g\int_0^t e^{\Phi(t)-\Phi(s)} \psi(s)\,ds\;,
$$
with $\Phi(t)=\int_0^t\varphi(s)ds$ and (a) is proven. Part (b)  follows immediately since $\psi(s)=0$
in the case  $b(t)=\bar b(t)$. Furthermore, if, in addition, $(\ell_0,m_0)=(\bar\ell_0,\overline m_0)$, then $$\eta(\ell(t),m(t)|\bar \ell(t),\overline m(t))=0\;,$$ for all $t\in[0,T]$ and by~\eqref{Sec3:relative eta2}, the two triplets $(\ell, m, n)$ and $(\bar \ell,\overline m, \bar n)$ coincide. Thus, by fundamental theorem of surfaces, the corresponding immersed maps are identical.
\ep

\begin{remark} 
1. The domain $\Omega$ is hypotheses $(H)$ and Theorem~\ref{newthm1} can be replaced by $\Omega=\RR\times [T_1,T_2]$. Then the results of the theorem still hold in the shifted time interval $[T_1,T_2]$. Now the analysis could be appropriately modified in the case that the domain $\Omega$ is bounded. However, since we aim to capture the global case $\RR^2$, we do not treat $\Omega=[-a,a]\times[0,T]$. Nevertheless, the above analysis can be adjusted for $x\in[-a,a]$, with $a>0$ or under perioodicity conditions in the space variable.  \\
2. We remark that hypothesis (H$_3$) allows us to compare the two metrics and this is reflected via the function $\psi$ given at~\eqref{S3: psi} that corresponds to the "distance" between the metrics $g$ and $\bar g$.
\end{remark}

\medskip

In the next proposition, we anticipate the global case.

\begin{theorem}[Global case]\label{newthm2} 
Suppose that $(\mathcal{M},g)$ is a geodesically complete simply connected two
dimensional Riemannian manifold with negative Gauss curvature $K$ and induced metric $g$ of the form~\eqref{S2 g} 
with an associated corrugated isometric immersion ${\bf y}$ in $\RR^3$, which is globally $C^{1,1}(\RR^2)$. Furthermore, suppose that $(\overline{\mathcal{M}},\bar g)$ is a geodesically complete simply connected smooth two
dimensional Riemannian manifold with negative Gauss curvature $\overline K$ and induced metric of the form~\eqref{S2 g} with an associated isometric immersion $\bar {\bf y}$ in $\RR^3$ that is smooth on $\RR^2$. Assuming that hypotheses (H$_0$)--(H$_2$) hold for $\Omega=\RR^2$ and assuming $g\equiv \bar g$, the followings hold true:
\begin{enumerate}
\item[(a)]   Stability in $L^2$ of the second fundamental form in terms of the initial data:
$$\int_{-\infty}^\infty |\ell(t)-\bar\ell(t)|^2+|m(t)-\overline m(t)|^2 dx\le C e^{\Phi(t)} \int_{-\infty}^\infty |\ell_0-\bar\ell_0|^2+|m_0-\overline m_0|^2 dx$$
for all $t\in \RR$ and some positive constant $C>0$ independent of the metrics and some bounded and positive functions $\Phi(s)$ on $\RR$, where $(\ell,m,n)$ and $(\bar\ell,\overline m, \bar n)$ are the corresponding solutions to~\eqref{SV:lm}--\eqref{SV:rel} with $(\ell_0,m_0)$ and $(\bar \ell_0, \overline m_0)$ are the corresponding initial data~\eqref{S2: conslaws ID}.
\item[(b)] Corrugated versus smooth uniqueness on $\RR^2$, i.e.: If $(\ell_0,m_0)=(\bar\ell_0,\overline m_0)$, then  the corrugated immersion ${\bf y}$ coincides with the smooth one $\bar {\bf y}$.
\end{enumerate}
\end{theorem}
\begin{proof}
The proof is a direct application of Theorem~\ref{newthm1} when $T=\infty$ and the fact that we can follow the analysis in its proof also backwards in time to cover also $t<0$. Then, by glueing the two parts, we get uniqueness on $\RR^2$. We note that a global immersion can be constructed by first showing existence of the hyperbolic problem~\eqref{SV:lm}--\eqref{SV:rel} for $t>0$ and then for $t<0$. The two solutions can be glued at $t=0$ if for instance the data of the second fundamental form are at least $C^1$. This strategy is also followed in~\cite{CS2015} and this still produces a corrugated immersion. In fact, discontinuities in  $(\bar\ell,\overline m, \bar n)$ arose in finite time since the data are not regular enough for Hong's result to be applied. However, the compensated compactness method in~\cite{CS2015} shows that $(\bar\ell (t),\overline m(t), \bar n(t))\in L^\infty(\RR)$ for all times.
\end{proof}

It follows now immediately as a direct application that if we compare the two immersions constructed by Hong~\cite{H} and Christoforou-Slemrod~\cite{CS2015}, we have uniqueness if the two metrics and the initial data at $t=0$ coincide.

\begin{corollary}\label{newthm3} 
Suppose that $(\mathcal{M},g)$ admits the global $C^{1,1}(\RR^2)$ isometric immersion ${\bf y}$ of Theorem~\ref{thm2}
and $(\overline{\mathcal{M}},\bar g)$ admits the global smooth isometric immersion $\bar {\bf y}$ of Theorem~\ref{HH}
with the same initial data. Then, for any $T>0$, the corrugated immersion ${\bf y}$ coincides with the smooth one $\bar {\bf y}$ on $\Omega=\mathbb{R}\times[-T,T]$.
\end{corollary}

The corollary follows immediately from Theorem~\ref{newthm2} since (H$_0$)--(H$_1$) are satisfied for $T=\infty$, while (H$_2$) for $T<\infty$.

\section{Examples}

One can apply the result of Theorem~\ref{newthm1} in various settings not necessarily the global case or for immersions with Gauss curvature decaying at Hong's rate $t^{-2-\delta}$. In general, one can compare immersed surfaces with different metrics or different initial data as long as the hypotheses of the theorem hold and at least one of the two surfaces is smooth in order for the relative entropy calculation to go through. 

We discuss here two applications both corresponding to same data:
\smallskip

\noindent
{\bf Example 1.} \emph{Comparing two helicoids.}\\
Consider two helicoids with different metrics and same data and one of them being smooth. Let
\begin{equation}\label{S4.3 helicoid1}
g_1=(1+t^2)d{x}^2+d{t}^2\;,
\end{equation}
and
\begin{equation}\label{S4.3 helicoidc}
g_c=\frac{(c^2+t^2)}{c^2}d{x}^2+d{t}^2
\end{equation}
be the two metrics corresponding to the helicoid-type surfaces $(\mathcal{H}_1, g_1)$ and  $(\mathcal{H}_c, g_c)$ with $c$ being a positive constant. Helicoid type surfaces have been studied in~\cite{CHW} in the setting of corrugated immersions and specifically, the metric written in the above form~\eqref{S4.3 helicoidc} is an example of the more general class of metrics considered in~\cite{CHW} using a \emph{different} change of variables from previous works~\cite{CSW, Cisom}. To clarify what it is captured in~\cite{CHW}, for instance for the metric~\eqref{S4.3 helicoidc}, is that if appropriate data are assigned at $t=-t_0<0$, then there exists an immersion close to the known helicoid surface described by the map
$${\bf y}(x,t)=(t\sin \frac{x}{c},\, t\cos \frac{x}{c}, \, x)
$$
and having discontinuous second derivatives $(h_{ij})=(\ell_c,m_c,n_c)$ in $L^\infty(\Omega)$, with $\Omega=\RR\times[-t_0,0]$, but with the same helicoid metric $g_c$. It should be remarked that, by even symmetry, one can obtain the weak solution in  $\Omega'=\left\{(x,t): x\in\mathbb{R}, 0\le t\le t_0\right\}$ and by periodicity it can be extended in $\RR^2$. 

Coming back to our question, we examine two such immersions and for convenience we take the case $c=1$ being smooth, while, for $c\ne 1$, we have an immersion not necessarily smooth as captured by~\cite{CSW}. We remark that both immersions correspond to geodesically complete manifolds and using our previous notation, we write
\begin{equation}
b_1(t)=\sqrt{1+t^2},\quad K_1(t)=-\displaystyle\frac{1}{(1+t^2)^2}\;,
\end{equation}
\begin{equation}
b_c(t)=\frac{1}{c}\sqrt{c^2+t^2},\quad K_c(t)=-\displaystyle\frac{c^2}{(c^2+t^2)^2}\;,
\end{equation}
with $K_1$ and $K_c$ the corresponding negative curvatures. From the analysis in~\cite{CHW}, it is shown that hypothesis (H$_0$) holds true, while (H$_1$) is trivially true. Then applying Theorem~\ref{newthm1} in the domain of the existence of the two immersions and using same data for their second fundamental forms, we get the following stability result:

\begin{proposition}
Let $T>0$ and $(\mathcal{H}_1, g_1)$ and  $(\mathcal{H}_c, g_c)$ two helicoid type surfaces corresponding to the metrics~\eqref{S4.3 helicoid1}--\eqref{S4.3 helicoidc} with $\mathcal{H}_1$ associated with the smooth map $\bar {\bf y}$ in $\RR^3$ while $\mathcal{H}_c$ with a locally $C^{1,1}$ map ${\bf y}$.
Then the corresponding second fundamental forms $(\ell_1,m_1,n_1)$ and $(\ell_c,m_c,n_c)$ are stable in $L^2(\RR)$ w.r.t. the parameter $c$ as follows
$$\int_{-\infty}^\infty |\ell_c(t)-\ell_1(t)|^2+|m_c(t)- m_1(t)|^2 dx\le L \,\frac{|c^2-1|}{c^2},\qquad t\in[0,T]$$
for some positive constant $L>0$ depending on $T$.
\end{proposition}
\begin{proof}
The result is a direct application of the analysis performed in Theorem~\ref{newthm1}. We estimate $\psi(s)$ given in~\eqref{S3: psi} replacing $\bar b$ and $\bar K$ by $b_1$ and $K_1$, while $b$ and $K$ are replaced by $b_c$ and $K_c$ according to the notation of this example. Then we have
$$
\psi(t)\approx \left |\frac{1}{c^2}-1\right| \frac{t^2}{\sqrt{1+t^2)(1+\frac{t^2}{c^2})}}+ \frac{17t^2}{(1+t^2)^2(c^2+t^2)^2} |c^2-1|^2 \lesssim A t^2 \frac{ |c^2-1|}{c^2}\,,
$$
for some constant $A$. Note that here $|c-1|$ is taken to be small and therefore the first order dominates. Now the factor $L$ is estimated after the integration in time is performed in the last term of part (a) in Theorem~\ref{newthm1} and it, therefore, depends on the final time $T$. We note that $c$ is present in $L$ via~\eqref{S3: phi}, but this does not affect the factor $|c^2-1|$ established in the continuous dependence.
\end{proof}

\noindent
{\bf Example 2.} \emph{Helicoid vs. Immersion with decay rate $t^{-2-\delta}$}\\
In this example, we discuss in short the comparison between a helicoid type surface and an immersion with the Hong's slow decay rate, that is $t^{-2-\delta}$. Note that the curvature for the helicoids in the previous example decay at the rate $t^{-4}$. Let the helicoid type surface $(\mathcal{H}_1, g_1)$ with the metric~\eqref{S4.3 helicoid1} discussed in the previous example and $(\mathcal{M}^*, g^*)$ with a metric
$ g^*=dt ^{2}+( b^*(x,t))^2 dx^{2}$ satisfying the assumptions of Theorem~\ref{thm2} with curvature
\be\label{SV: kexample} |K^*|=\frac{C}{(1+|t|)^{2+\delta}},\quad \text{for all }|t|>T_1\ee
for some $T_1$ large enough. We can compare the corresponding second fundamental forms
$(\ell_1,m_1,n_1)$ and $(\ell^*,m^*,n^*)$ in terms of the ``distance" between the two metrics $g_1$ and $g^*$ via
$$\int_{-\infty}^\infty |\ell_1(t)-\ell^*(t)|^2+|m_1(t)-m^*(t)|^2 dx\le C_g \int_0^T e^{\Phi(t)-\Phi(s)} \psi(s)ds  
$$
from Theorem~\ref{newthm1} for $t\in[0,T]$. In fact, here we assume that both immersions admit same initial data at $t=0$ and therefore, the difference between the two forms would arise only due to the different metrics.

From Han and Hong~\cite[Lemma 10.2.3]{HH}, we have that if $|K^*(s)|$ and $s|K^*(s)|$ are in $L^1(0,\infty)$, then
$$ 1+\int_0^t\int_0^s |K^*(\tau)|dt\tau ds\le b^*(t)\le 1+C_1 t\,.$$
Moreover, if $|K^*(t)| t^{2+\delta}$ is decreasing in $|t|$ for $|t|>T_1$, then
$$ \partial_t \ln b^*=\frac{1}{t}+O(\frac{1}{|t|^{1+\delta}})\;,$$
for sufficiently large $|t|$ and $\partial_t\ln b^*$ is bounded. Using these, we have the following asymptotic behavior of the $\psi$ given in~\eqref{S3: psi}:
\begin{align*}\label{S3: psiex2}
\psi(t):
&=\left| \frac{t}{{1+t^2}} - \del_t\ln {b^*(t)}\right|^2+\left|\del_t\ln {(1+t^2)^2}+  \del_t\ln{|K^*(t)|}\right|^2+ \left|\frac{1}{b^*(t)}-\frac{1}{\sqrt{1+t^2}}\right|\nonumber\\
&\approx\left| \frac{t}{{1+t^2}} - (\frac{1}{t}+O(\frac{1}{|t|^{1+\delta}}))\right|^2+ \left| \frac{4t}{{1+t^2}}- \frac{2+\delta}{t}\right|^2 +\left|\frac{1}{1+t^{\delta}}-\frac{1}{\sqrt{1+t^2}}\right|\,,
\end{align*}
for $|t|>T_1$.

\end{document}